\documentclass[12pt]{amsart}
\usepackage{amsmath}
\usepackage{amssymb}
\usepackage[headings]{fullpage}
\usepackage{amsfonts}
\usepackage{paralist}
\usepackage{xspace}
\usepackage{euscript}
\usepackage{color,tikz}
\usetikzlibrary{decorations.pathreplacing,calc}

\usepackage{epigraph}
\usepackage[all]{xy}
\usepackage[normalem]{ulem}
\usepackage[colorlinks,pagebackref=true,pdftex]{hyperref}
\usepackage[alphabetic,backrefs,abbrev,lite]{amsrefs} 

\makeatletter
\@addtoreset{equation}{section}
\def\theequation{\thesection.\@arabic \c@equation}

\def\theenumi{\@roman\c@enumi}
\def\@citecolor{blue}
\def\@linkcolor{blue}
\def\@urlcolor{blue}

\makeatother

\newtheorem{theorem}[equation]{Theorem}
\newtheorem{lemma}[equation]{Lemma}
\newtheorem{proposition}[equation]{Proposition}

\newtheorem{conjecture}[equation]{Conjecture}

\newtheorem{claim*}{Claim}

\newtheorem{question}[equation]{Question}

\theoremstyle{definition}
\newtheorem{rmk}[equation]{Remark}
\newenvironment{remark}[1][]{%
    \begin{rmk}[#1] \pushQED{\qed}}{\popQED \end{rmk}}

\newtheorem{eg}[equation]{Example}
\newenvironment{example}[1][]{%
    \begin{eg}[#1] \pushQED{\qed}}{\popQED \end{eg}}

\newtheorem{definition}[equation]{Definition}

\newtheorem{notn}[equation]{Notation}

\def\<{\langle}
\def\>{\rangle}

\newcommand{\BQ}{\mathrm{B}_{\QQ}}

\newcommand{\BQtot}{\overline{\mathrm{B}_{\QQ}(R_{\infty})}}

\newcommand{\rank}{\operatorname{rank}}

\renewcommand{\to}{\longrightarrow}

\newcommand{\fm}{\mathfrak m}

\newcommand{\QQ}{\mathbb{Q}}

\newcommand{\VV}{\mathbb{V}}
\newcommand{\WW}{\mathbb{W}}

\newcommand{\ZZ}{\mathbb{Z}}

\newcommand{\edim}{\operatorname{edim}}
\newcommand{\rhofree}{\rho_{-1}}
\newcommand{\defi}[1]{{\bfseries\upshape #1}}

\newcommand{\minus}{\ensuremath{\!\smallsetminus\!}}

\title{Shapes of free resolutions over a local ring}

\begin{document}

\author[C. Berkesch]{Christine Berkesch}
\address{Institut Mittag-Leffler \\ Aurav\"agen 17 \\ 
SE-182 60 Djursholm, Sweden \hfill \quad \linebreak\vspace{-.7cm}}
\address{
Department of Mathematics \\ Stockholm University \\
SE-106 91 Stockholm, Sweden}
\curraddr{Department of Mathematics, Duke University, Box 90320, Durham, NC 27708. }
\email{cberkesc@math.duke.edu}

\author[D. Erman]{Daniel Erman}
\address{Department of Mathematics \\ Stanford University \\
Stanford, CA 94305}
\curraddr{Department of Mathematics, University of Michigan, Ann Arbor, MI
48109.}
\email{erman@umich.edu}

\author[M. Kummini]{Manoj Kummini}
\address{Department of Mathematics \\ Purdue University \\
West Lafayette, IN 47907}
\curraddr{Chennai Mathematical Institute, Siruseri, Tamilnadu 603103, India. }
\email{mkummini@cmi.ac.in}

\author[S. Sam]{Steven V Sam}
\address{Department of Mathematics \\ Massachusetts Institute of
Technology \\
Cambridge, MA 02139}
\email{ssam@math.mit.edu}

\thanks{CB was supported by NSF Grant OISE 0964985, DE by an NDSEG
  fellowship and NSF Award No.~1003997, and SVS by an NSF graduate
  research fellowship and an NDSEG fellowship. }


\begin{abstract}
We classify the possible shapes of minimal free 
resolutions over a regular local ring.  This illustrates the existence of free 
resolutions whose Betti numbers behave in surprisingly pathological ways.  
We also give an asymptotic characterization of the
possible shapes of minimal free resolutions over hypersurface rings.
Our key new technique uses asymptotic arguments to study formal 
$\QQ$-Betti sequences.
\end{abstract}

\maketitle

\section{Introduction}
\label{sec:intro}

Let $M$ be a finitely generated module over a local ring $R$. 
From its minimal free resolution
\[
0 \gets M \gets R^{ b_0}\gets R^{ b_1}\gets R^{ b_2}\gets \cdots
\]
we obtain the \defi{Betti sequence} $b^R(M):=(b_0,b_1,b_2, \dots)$ of
$M$.  Questions about the possible behavior of $b^R(M)$ arise in many
different contexts (see \cite{peeva-stillman} for a recent survey).
For instance, the Buchsbaum--Eisenbud--Horrocks Rank Conjecture
proposes lower bounds for each $b^R_i(M)$, at least when $R$ is
regular, and this conjecture is related to multiplicative structures
on resolutions~\cite{buchs-eis-gor}*{p.\ 453}, vector bundles on
punctured discs~\cite{hartshorne-vector}*{Problem~24}, and equivariant
cohomology of products of spheres (\cite{carlsson2}
and~\cite{carlsson}*{Conj~II.8}). When $R$ is not regular, there are even
more
questions about the possible behavior of
$b^R(M)$~\cite{avramov-notes}*{\S4}.

Here we consider the qualitative behavior of these sequences; we
define the \defi{shape} of the free resolution of $M$ as the Betti
sequence $b^R(M)$ viewed \emph{up to scalar multiple}. Instead of
asking if there exists a module $M$ with a given Betti sequence,
say $\mathbf{v}=(18,20,4,4,20,18)$, we ask if there exists a
Betti sequence $b^R(M)$ with the same shape as $\mathbf{v}$, i.e.,
whether $b^R(M)$ is a scalar multiple of $\mathbf{v}$.  In a sense,
this approach is orthogonal to questions like the
Buchsbaum--Eisenbud--Horrocks Rank Conjecture,
which focus on the {\em size} of a free resolution.

In this article, we show that this shift in approach, which was
motivated by ideas of~\cite{boij-sod1}, provides a clarifying
viewpoint on Betti sequences over local rings. First, we completely
classify shapes of resolutions when $R$ is regular.  To state the
result, we let $\mathbb V= \QQ^{n+1}$ be a vector space with standard basis
$\{\epsilon_i\}_{i=0}^n$.

\begin{theorem}\label{thm:regmain}
Let $R$ be an $n$-dimensional regular local ring, 
$\mathbf{v} := (v_i)_{i=0}^n \in \mathbb V$, and $0 \leq d \leq n$. 
Then the following are equivalent:
\begin{enumerate}[\rm (i)]
\item\label{thm:regmain:mod} There exists a finitely generated
  $R$-module $M$ of depth $d$ such that $b^R(M)$ has shape
  $\mathbf{v}$, i.e., there exists $\lambda\in \QQ_{>0}$ such that
  $b^R(M)=\lambda \mathbf{v}$.
\item\label{thm:regmain:sum}  
	There exist $a_{-1}\in\mathbb Q_{\geq 0}$ and $a_i\in \mathbb Q_{>0}$ for $i\in\{0, \dots, n-d-1\}$ such that
	\[
    \mathbf{v} = a_{-1}\epsilon_0 + 
    \sum_{i=0}^{n-d-1} a_i(\epsilon_i+\epsilon_{i+1}).
	\]
\end{enumerate}
If $a_{-1}=0$ in \eqref{thm:regmain:sum}, 
then $M$ can also be chosen to be Cohen--Macaulay.
\end{theorem}

This demonstrates that there are almost no bounds on the shape of 
a minimal free $R$-resolution. While showing 
that~\eqref{thm:regmain:mod} implies \eqref{thm:regmain:sum} is
straightforward, the converse is more interesting, as it leads
to examples of free resolutions with unexpected behavior.
For instance, let $R=\QQ[[x_1,\dots,x_{14}]]$, fix some $0 < \delta \ll 1$, 
and let $\mathbf{v}=(1-\frac{\delta}{2}, 1, \delta, \delta, \delta, \delta,
\delta, \delta, 4, 4, \delta, \delta, \delta, 1, 1-\frac{\delta}{2})$.
Plotting its entries, the shape of $\mathbf{v}$ is shown in Figure~\ref{fig:v}. 
As $\mathbf{v}$ satisfies
Theorem~\ref{thm:regmain}\eqref{thm:regmain:sum}, 
there exists a finite length $R$-module $M$ whose minimal free resolution 
has this shape. 
Similar pathological examples abound.
\begin{figure}
\vspace*{.65cm}
\begin{tikzpicture}
\filldraw (0,.63) circle (1pt);
\filldraw (.4,.7) circle (1pt);
\filldraw (.8,.1) circle (1pt);
\filldraw (1.2,.1) circle (1pt);
\filldraw (1.6,.1) circle (1pt);
\filldraw (2,.1) circle (1pt);
\filldraw (2.4,.1) circle (1pt);
\filldraw (2.8,.1) circle (1pt);
\filldraw (3.2,2.8) circle (1pt);
\filldraw (3.6,2.8) circle (1pt);
\filldraw (4,.1) circle (1pt);
\filldraw (4.4,.1) circle (1pt);
\filldraw (4.8,.1) circle (1pt);
\filldraw (5.2,.7) circle (1pt);
\filldraw (5.6,.63) circle (1pt);
\draw[thick] (0,.63) --(.4,.7) --(.8,.1) --(1.2,.1) --(1.6,.1) --(2.0,.1)
--(2.4,.1) --(2.8,.1) --
 (3.2,2.8) --(3.6,2.8) --(4,.1) --(4.4,.1) --
 (4.8,.1) --(5.2,.7) --(5.6,.63);
\draw[thin,dashed](0,2.8)--(0,0)--(5.7,0);
\draw (2.8,-.5) node {$i$};
\draw (-.5,1.4) node {$b_i(M)$};
\end{tikzpicture}
\hspace{1cm}
\begin{tikzpicture}
\filldraw (0,1.9) circle (1pt);
\filldraw (.4,2) circle (1pt);
\filldraw (.8,.1) circle (1pt);
\filldraw (1.2,2) circle (1pt);
\filldraw (1.6,2) circle (1pt);
\filldraw (2,.1) circle (1pt);
\filldraw (2.4,2) circle (1pt);
\filldraw (2.8,2) circle (1pt);
\filldraw (3.2,.1) circle (1pt);
\filldraw (3.6,2) circle (1pt);
\draw (4,.2) node {\dots};
\filldraw (4.4,2) circle (1pt);
\filldraw (4.8,.1) circle (1pt);
\filldraw (5.2,2) circle (1pt);
\filldraw (5.6,1.9) circle (1pt);
\draw[thick] (0,1.9) -- (.4,2) -- (.8,.1)
--(1.2,2)--(1.6,2)--(2,.1)--(2.4,2)--(2.8,2)--(3.2,.1)--(3.6,2);
\draw[thick] (4.4,2) -- (4.8,.1) -- (5.2,2) --(5.6,1.9);
\draw[thin,dashed](0,2.8)--(0,0)--(5.7,0);
\draw (2.8,-.5) node {$i$};
\draw (-.5,1.4) node {$b_i(N)$};
\end{tikzpicture}
\caption{On the left, we illustrate the shape of 
$\mathbf{v}=(1-\frac{\delta}{2}, 1, \delta, \delta, \delta, \delta,
\delta, \delta, 4, 4, \delta, \delta, \delta, 1, 1-\frac{\delta}{2})$
where $0<\delta\ll 1$ is a rational number.
On the right, we illustrate an oscillating shape,
as in Example~\ref{ex:oscillate}.  
Each arises as the shape of some minimal free resolution.}
\label{fig:v}
\end{figure}
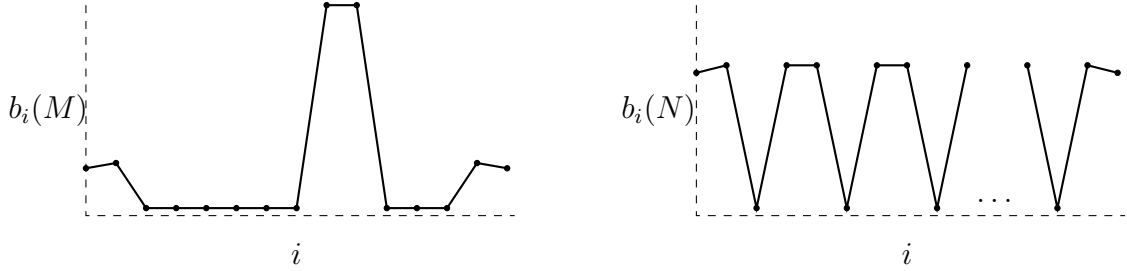
As mentioned above, our work is inspired by the Boij--S\"oderberg 
perspective that the numerics of minimal free resolutions over a graded 
polynomial ring $S$ are easier to understand if one works up to scalar multiple. 
They introduced the cone of Betti diagrams over $S$ and provided 
conjectures about the structure of this cone.  Their conjectures were proven and 
extended in a series of papers~\cites{boij-sod1,boij-sod2,efw,ES-JAMS}. 
(See also~\cite{ES:ICMsurvey} for a survey.)

To provide a local version of Boij--S\"oderberg theory, we study the
\defi{cone of Betti sequences} $\BQ(R)$, which we define to be 
the convex cone spanned by all points $b^R(M)\in \VV$, 
where $M$ is a finitely generated $R$-module.  
Theorem~\ref{thm:regmain} implies that the closure of $\BQ(R)$ 
is spanned by the rays corresponding to $\epsilon_0$ and 
$(\epsilon_i+\epsilon_{i+1})$ for $i=0,\dots, n-1$.  
The point $(\epsilon_i+\epsilon_{i+1})$ can be
interpreted as the Betti sequence of the non-minimal complex
$(R^1\stackrel{\sim}{\longleftarrow}R^1)$, where the copies of $R$ lie in
homological positions $i$ and $i+1$. Since this is not itself a
minimal free resolution, it follows that $\BQ(R)$ is not a closed
cone, in contrast with the graded case. 
The facet equation description of $\BQ(R)$ is also simpler than in the graded case: by Proposition~\ref{prop:moredetailed} 
below, all facets are given by partial Euler characteristics.

Our proof of Theorem~\ref{thm:regmain} relies on a limiting 
technique that is possible because we study Betti sequences 
in $R$ only up to scalar multiple; the introduction of the 
rational points of $\BQ(R)$, which can be thought of as 
formal $\QQ$-Betti sequences, enables the use of this technique. 
To produce the necessary limiting sequences, we first produce 
local analogues of the Eisenbud--Schreyer pure resolutions, 
as we have precise control over their Betti numbers. 

We emphasize here the fact that $\BQ(R)$ depends only on 
the dimension of $R$.  In particular, the result is the same 
for both equicharacteristic and mixed characteristic rings. 

\subsection*{Hypersurface rings}
We also examine the shapes of
minimal free resolutions over the simplest singular local rings:
hypersurface rings.  
Given a regular local ring $(R,\mathfrak m_R)$, 
we say that $Q$ is a \defi{hypersurface ring of} $R$ 
if $Q=R/\<f\>$ and $f\in \mathfrak m_R^2$.

Unlike the regular local case, free resolutions are not necessarily
finite in length over a hypersurface ring.  Hence Betti sequences
$b^Q(M)$ lie in an infinite dimensional vector space
$\WW:=\prod_{i=0}^\infty \QQ$.  We let $\{\epsilon_i\}$ denote the
coordinate vectors of $\WW$ and we write elements of $\WW$ as possibly
infinite sums $\sum_{i=0}^\infty a_i\epsilon_i$. We also view $\VV$ as
a subspace of $\WW$ in the natural way.

The key tool for studying free resolutions over a hypersurface ring is
the \defi{standard construction} (which is briefly reviewed in
\S\ref{sec:cone:all:hyp}). Given a $Q$-module $M$, this builds a
(generally non-minimal) $Q$-free resolution of $M$ from the minimal
$R$-free resolution of $M$.  The numerics of this free resolution of
$M$ are easy to understand in terms of $b^R(M)$.  Define $\Phi \colon
\WW\to \WW$ by
\[
\Phi(v_0,v_1,v_2,\dots):=(v_0, v_1, v_0+v_2,v_1+v_3,v_0+v_2+v_4,\dots).
\]
The standard construction for $M$ yields a (generally non-minimal)
resolution $G_\bullet$ with Betti sequence
$b^Q(G_\bullet)=\Phi(b^R(M))$.

Due to this close connection between free resolutions over $R$ and
over $Q$, it is tempting to conjecture that the numerics of $\BQ(Q)$
should be controlled by the cone $\BQ(R)$ and the map $\Phi$.
However, additional ingredients are clearly required.  First, the
sequence $\Phi(b^R(M))$ always has infinite length, whereas there do
exist minimal free resolutions over $Q$ with finite projective
dimension.  Second, if an $R$-module $M$ is annihilated by some
polynomial $f$, then it automatically has rank $0$ as an
$R$-module. Thus we should only be interested in applying $\Phi$ to
modules of rank $0$.

The following theorem shows that 
all minimal free resolutions over hypersurface 
rings of $R$ are controlled by correcting precisely these two factors.  

\begin{theorem}\label{thm:hypmain}
Let $(R, \mathfrak{m}_R)$ be an $n$-dimensional regular local ring,
let $\overline R$ be an $(n-1)$-dimensional regular local ring, 
and fix $\mathbf{w} := (w_i)_{i=0}^\infty \in \WW$. 
Then the following are equivalent:
\begin{enumerate}[\rm (i)]
\item\label{thm:hypmain:mod} There exists $f\in \mathfrak
  m_R$, a positive integer $\lambda$, and a finitely generated $R/\<f\>$-module
  $M$ such that $b^{R/\<f\>}(M)=\lambda
  \mathbf{w}$.
\item\label{thm:hypmain:sum} 
	There exists an $R$-module $M_1$ of  rank $0$ 
	and an $\overline{R}$-module $M_2$ such that
$
\mathbf{w}=\Phi(b^R(M_1))+b^{\overline{R}}(M_2).
$
\end{enumerate}
\end{theorem}

This demonstrates that, except for eventual periodicity, 
there are essentially no bounds on the shape of a minimal free resolution 
over a hypersurface ring of $R$. As in the regular local case, this 
leads to examples of free resolutions with surprising behavior.
For instance, fix any $\delta>0$ and let $R=\QQ[[x_1,\dots,x_{14}]]$.
Applying Theorem~\ref{thm:regmain}, there exist $M_1$ and $M_2$ 
so that 
$\mathbf{w}=\Phi(b^R(M_1))+b^{\overline{R}}(M_2)$, where
\[
\mathbf{w}:=(\tfrac{\delta}{2},4,4,\delta,\delta,\delta,\delta,\delta,
\delta,\delta,\delta,1,1,\delta,6+\tfrac{\delta}{2},6,6,6,\dots).
\]  
Since $\mathbf{w}$ satisfies
Theorem~\ref{thm:hypmain}\eqref{thm:hypmain:sum}, there exists a
module $M$ over a hypersurface ring of $R$ whose
minimal free resolution has this shape. 

We now make the connection with local Boij--S\"oderberg theory explicit.

\vspace{.25mm}

\begin{definition}\label{defn:Rtot}
The \defi{total hypersurface cone} $\BQtot$ is the closure in
  $\WW$ of the union $\bigcup_{f \in \mathfrak m_R} \BQ(R/\<f\>)$.
\end{definition}

\vspace{.35mm}
We show in Remark~\ref{rmk:equals asymptotic} 
that the cone $\BQtot$ may also be realized as a limit of cones
\begin{equation}\label{eqn:equals asymptotic}
\BQtot = \lim_{t\to \infty} \BQ(R/\<f_t\>)\subseteq \WW
\end{equation}
for any sequence $(f_t \in \mathfrak m^t_R)_{ t \geq 1}$. 

The following result provides an 
extremal rays description of this cone.

\begin{proposition}\label{prop:hypersurface rays}
The cone $\BQtot$ is an $(n+1)$-dimensional subcone of $\WW$ 
spanned by the following list of $(n+2)$ extremal rays:
\begin{enumerate}[\rm (i)]
\item the ray spanned by $\epsilon_0$,
\item the rays spanned by $(\epsilon_i+\epsilon_{i+1})$ for
  $i\in\{0, \dots, n-2\}$, and 
\item  the rays spanned by
  \[
  \sum_{i=n-2}^\infty \epsilon_i \quad \text{ and }\quad
  \sum_{i=n-1}^\infty \epsilon_i.
  \]
\end{enumerate}
\end{proposition}

\smallskip
The proofs of Theorem~\ref{thm:hypmain} and Proposition~\ref{prop:hypersurface rays}
rely on two types of asymptotic arguments.  First, as in the proof of Theorem~\ref{thm:regmain},
we study sequences of formal $\QQ$-Betti sequences.  Second, we use
that the cone $\BQtot$ is itself a limit, as illustrated in \eqref{eqn:equals asymptotic}.

In Proposition~\ref{prop:moreDetailedHyper}, we also describe the cone
$\BQtot$ in terms of defining hyperplanes.  In addition, we observe
that, as in the description of $\BQ(R)$, most of the extremal rays of
$\BQtot$ do not correspond to actual minimal free resolutions.  Note
that, based on \eqref{eqn:equals asymptotic}, the cone $\BQ(R/\<f\>)$
is closely approximated by $\BQtot$, at least when the Hilbert--Samuel
multiplicity of $R/\<f\>$ is large.

We end by considering the more precise question of completely describing 
$\BQ(R/\<f\>)$ for a fixed $f\in \mathfrak m_R$. 
The following conjecture claims that the cone $\BQ(R/\<f\>)$ depends 
only on the dimension and multiplicity of the hypersurface ring $R/\<f\>$.

\begin{conjecture}\label{conj:hyper:equiv}
Let $Q$ be a hypersurface ring of embedding dimension $n$ 
and multiplicity $d$. 
Then $\BQ(Q)$ is an $(n+1)$-dimensional cone, and its closure is defined 
by the following $(n+2)$ extremal rays:
\begin{enumerate}[\rm (i)]
\item the ray spanned by $\epsilon_0$, 
\item the rays spanned by $(\epsilon_i+\epsilon_{i+1})$ 
	for $i=\{0, \dots, n-2\}$, and 
\item the rays spanned by
  \[
  \tfrac{d-1}{d}\epsilon_{n-2}+\sum_{i=n-1}^\infty \epsilon_i \quad
  \text{ and }\quad \tfrac{1}{d}\epsilon_{n-2}+\sum_{i=n-1}^\infty
  \epsilon_i.
	\]
\end{enumerate}
\end{conjecture}

Proposition~\ref{prop:one direction} proves one direction of this
conjecture, by showing that $\BQ(Q)$ belongs to the cone spanned by
the proposed extremal rays. We also prove
Conjecture~\ref{conj:hyper:equiv} when $\edim(Q)=2$.  Observe also
that Proposition~\ref{prop:hypersurface rays} is essentially the
$d=\infty$ version of this conjecture.
%

\subsection*{Notation}
Throughout the rest of this document $R$ will be a regular local ring
and $Q$ will be a quotient ring of $R$.  If $M$ is an $R$-module or a
$Q$-module, then $e(M)$ is the Hilbert--Samuel multiplicity of $M$ and
$\mu(M)$ is the minimal number of generators for $M$.  Given a
surjection $R^{\mu(M)} \to M$, we denote the kernel by $\Omega(M)$,
and in general, we set $\Omega^{j}(M) = \Omega^1(\Omega^{j-1}(M))$,
with the convention $\Omega^0(M) = M$, and we call $\Omega^j(M)$ the
{\bf $j$th syzygy module} of $M$.

\subsection*{Acknowledgements}
Significant parts of this work were done when the second author
visited Purdue University and during a workshop funded by the Stanford
Mathematics Research Center; the paper was completed while the first 
author attended the program ``Algebraic Geometry with a view towards
applications" at Institut Mittag-Leffler; we are grateful for all of
these opportunities.  Throughout the course of this work, calculations
were performed using the software {\tt Macaulay2}~\cite{M2}.  We thank
Matthias Beck for pointing out the reference~\cite{triangulations}.
We also thank Jesse Burke, David Eisenbud, Courtney Gibbons, 
Mel Hochster, Frank-Olaf Schreyer, and Jerzy Weyman for insightful 
conversations.  
We also thank the referee for suggestions that greatly streamlined this paper. 

\section{Passage of graded pure resolutions to a regular local ring}
\label{sec:pass:pure}

To prove Theorem~\ref{thm:regmain}, we produce a collection of Betti
sequences that converge to each extremal ray of $\overline{\BQ(R)}$.
The key step in constructing these sequences is the construction of
local analogues of the pure resolutions of Eisenbud and Schreyer.

Let $S=\ZZ[x_1, \dots, x_n]$.  Fix $d=(d_0,\dots,d_s)\in \ZZ^{s+1}$
with $d_i<d_{i+1}$ and $s\leq n$.  By~\cite{BEKStensor}*{Remark~10.2}
and~\cite{ES-JAMS}*{\S5}, we may construct an $S$-module $M(d)$ that
is a generically perfect $S$-module of
codimension $s$ (and hence, $M(d)\otimes_{\ZZ} \Bbbk$ is a Cohen--Macaulay
module of codimension $s$ for every field $\Bbbk$.)

\begin{proposition}\label{prop:pureoverRLR}
  Let $R$ be an $n$-dimensional regular local ring. 
  Let $S\to R$ be any map sending $x_1, \dots, x_n$ to an 
  $R$-regular sequence.  
  Then $M(d)\otimes_{S} R$ is a Cohen--Macaulay $R$-module of codimension $s$, 
  and the Betti sequence of $M(d)\otimes_S R$ is a scalar multiple of
\[
\mathbf{v}(d):=\left(\frac{1}{\prod_{i\ne 0} |d_i-d_0|},
\frac{1}{\prod_{i\ne 1} |d_i-d_1|}, \dots, \frac{1}{\prod_{i\ne s}
|d_i-d_s|},0,\dots,0\right)\in \VV.
\]
\end{proposition}
\begin{proof}
We have noted above that $M(d)$ is a generically
perfect $S$-module of codimension $s$.  It follows 
from~\cite{bruns-vetter}*{Theorem~3.9} that
$M(d)\otimes_S R$ is Cohen--Macaulay 
and of the same codimension as $M(d)$.  In addition, by
\cite{bruns-vetter}*{Theorem~3.5}, tensoring a 
minimal $S$-free resolution of $M(d)$ with $R$ gives 
a minimal $R$-free resolution of
$M(d)\otimes_S R$.
The formula for $\mathbf{v}(d)$ then follows from the Herzog--K\"uhl
equations~\cite {HerzogKuhlPure84}*{Theorem~1}.
\end{proof}
  
\section{Cone of Betti sequences for a regular local ring}
\label{sec:cone:reg}

Let $(R, \fm)$ be an $n$-dimensional regular local ring. 
Let $\VV := \bigoplus_{i=0}^n \QQ$, with basis $\{\epsilon_i\}$, 
where $0 \leq i \leq n$. 
For $i = 0, \ldots, n-1$, set $\rho_i := \epsilon_i+\epsilon_{i+1}$, 
and set $\rhofree:=\epsilon_0$. 
For all $i\leq j$, we define the partial Euler characteristic functionals 
\begin{align*}
\chi_{[i,j]}
	:=&\, \epsilon^*_i-\epsilon^*_{i+1}+\dots+(-1)^{j-i}\epsilon^*_j\\
	=&\sum_{\ell=i}^j(-1)^{\ell-i}\epsilon_\ell^*.
\end{align*}
For a ring $R$, we set $\overline{\BQ(R)}$ to be the closure of the
cone $\BQ(R)\subseteq \VV$, which we describe now. 

\begin{proposition}\label{prop:moredetailed}
For any $n$-dimensional regular local ring $R$, 
the following three $(n+1)$-dimensional cones are equal: 
\begin{enumerate}[\rm (i)]
\item \label{enum:RLRConeClosure}
	the closure $\overline{\BQ(R)}$ of the cone of Betti sequences.
\item \label{enum:RLRraySpan}
	the cone spanned by the rays 
	$\QQ_{\geq 0}\langle \rhofree, \rho_0, \rho_1, \dots, 
		\rho_{n-1}\rangle$. 
\item \label{enum:RLRfacets} 
	the intersection of the halfspaces defined by 
	$\chi_{[j,n]}\geq 0$ for $j\in\{0, \dots, n\}$.
\end{enumerate}
\end{proposition}
\begin{proof}
The work here lies in showing that \eqref{enum:RLRraySpan} 
is contained in \eqref{enum:RLRConeClosure}; 
this is where we use a limiting argument. 
We first verify the straightforward containments. 
The rays of \eqref{enum:RLRraySpan} satisfy the 
inequalities of \eqref{enum:RLRfacets} because
\begin{equation*}\label{eqn:FunctionalsOnRhos}
\chi_{[j,n]}(\rho_i)  = \begin{cases}
0&\text{if } j\ne i+1,\\
1&\text{if } j=i+1.
\end{cases}
\end{equation*}
Conversely, if $\mathbf{v}\in \VV$ satisfies all of the inequalities, 
then we can write
$
\mathbf{v}= \sum_{i=-1}^{n-1}
\chi_{[i+1,n]}(\mathbf{v})\cdot\rho_{i}, 
$ 
which lies in \eqref{enum:RLRraySpan}.  So we have shown the
equivalence of \eqref{enum:RLRraySpan} and \eqref{enum:RLRfacets}.

To see that the functionals of \eqref{enum:RLRfacets} 
are nonnegative on $\overline{\BQ(R)}$,
it suffices to consider a point of the form $b^R(M)$.  In this
case, $\chi_{[i,n]}(b^R(M)) = \rank \Omega^i(M)$ for $i\geq 0$. 
This implies that $\BQ(R)$ lies in \eqref{enum:RLRfacets}, 
and hence so does its closure.

It thus suffices to check that the rays $\rho_i$
in~\eqref{enum:RLRraySpan} belong to $\overline{\BQ(R)}$.  Since
$\rhofree=\beta(R^1)$, we have $\rhofree \in \BQ(R)$.  To show that
$\rho_j\in \overline{\BQ(R)}$ for $j\geq 0$, we use a limiting
argument.  Such an argument is necessary because the vectors $\rho_j$
do not belong to $\BQ(R)$ due to their non-minimal structure (at least
when $j>0$).  Adopt the notation of \S\ref{sec:pass:pure} and define
$\mathbf{v}_j(d)$ to be the unique scalar multiple of $\mathbf{v}(d)$
such that $\mathbf{v}(d)_j=1$.  Based on the formula for
$\mathbf{v}(d)$ from Proposition~\ref{prop:pureoverRLR}, view
$\mathbf{v}_j$ as a map from $\ZZ^{n+1}\to \VV$ (with poles) defined
by the formula
\[
\mathbf{v}_j(d_0,\dots,d_n)=\left(\frac{\prod_{i\ne j}
|d_i-d_j|}{\prod_{i\ne 0} |d_i-d_0|}, \frac{\prod_{i\ne j}
|d_i-d_j|}{\prod_{i\ne 1} |d_i-d_1|}, \dots, \frac{\prod_{i\ne j}
|d_i-d_j|}{\prod_{i\ne n} |d_i-d_n|}\right)\in \VV.
\]

And now for the crucial choice, 
which is explored further in Example~\ref{ex:djk}. 
For each $j$, consider the sequence $\{ d^{j,t}\}_{t\geq0}$ defined by
$d^{j,t}:=(0, t, 2t, \ldots, jt, jt+1,(j+1)t+1, \ldots, (n-1)t+1)$.
In other words, 
\[
d^{j,t}_k=\begin{cases}
kt &\text{if $k\leq j$,}\\
(k-1)t+1 &\text{if $k>j$.} 
\end{cases}
\]
We claim that
$
\rho_j = \lim_{t\to \infty} \mathbf{v}_j(d^{j,t}). 
$ 
This would imply, by Proposition~\ref{prop:pureoverRLR}, that
$\rho_i\in \overline{\BQ(R)}$, thus completing the proof.
To prove this claim, we observe that the $j$th coordinate function of
$\mathbf{v}_j$ equals $1$ and $\mathbf{v}_j(d)$ lies in the hyperplane
defined by $\chi_{[0,n]}=0$.  So it suffices to prove that the
$\ell$th coordinate function of $\mathbf{v_j}$ goes to $0$ for all
$\ell\ne j,j+1$.  We directly compute
\begin{align*}
  \lim_{t\to \infty} \mathbf{v}_j(d^{j,t})_{\ell}&=\lim_{t\to
    \infty}\frac{\prod_{i\ne j} |d^{j,t}_i-d^{j,t}_j|}{\prod_{i\ne
      \ell} |d^{j,t}_i-d^{j,t}_\ell|}=\lim_{t\to
    \infty}\frac{O(t^{n-1})}{O(t^{n})} =0. \qedhere
\end{align*} 
\end{proof}

\begin{example}\label{ex:djk}
If $n = 4$, then 
$d^{1,t} = (0,t,t+1,2t+1,3t+1)$.
Over $S = \Bbbk [x_1,\dots,x_4]$ with the standard grading, 
this degree sequence corresponds to the Betti diagram 
\[
\beta^S(M(d^{1,t})) = 
\begin{bmatrix}
\beta_0^{1,t} & - & - & - & - \\
\vdots & \vdots & \vdots & \vdots & \vdots\\
- & \beta_1^{1,t} & \beta_2^{1,t} & - & - \\
\vdots & \vdots & \vdots & \vdots & \vdots \\
- & - & - & \beta_3^{1,t} & - \\
\vdots & \vdots & \vdots & \vdots & \vdots \\
- & - & - & - & \beta_4^{1,t}
\end{bmatrix}
\hspace{-.8cm}
\mbox{
\scriptsize
\begin{tabular}{l}
	$\left.\begin{tabular}{l}
	 \ \\ \ \\ \ 
	\end{tabular}\right\}$ 
	\\
	 $\left.\begin{tabular}{l}
	\ \\ \ \\ \ 
	\end{tabular}\right\}$ 
	\\ 
	 $\left.\begin{tabular}{l}
	\ \\ \ \\ \ 
	\end{tabular}\right\}$ 
	\\
	 $\begin{tabular}{l}
	\ \\ 
	\end{tabular}$ 
\end{tabular}
}
\hspace{-.5cm}
\mbox{
\begin{tabular}{l}
\ \vspace{.2cm} \\
$t-1$ rows \\ \vspace{.15cm} \\
$t-1$ rows \\ \vspace{.15cm} \\
$t-1$ rows \\ \\
\ 
\end{tabular}
}
\]
where there are gaps of $t-3$ rows of zeroes between the various 
nonzero entries.
Notice that as $t\to\infty$, this Betti diagram gets longer.  
It is thus necessary to consider the total Betti 
numbers $\beta_{i}$ (i.e., to forget about the individual graded Betti numbers 
$\beta_{i,j}$) before it makes sense to 
consider a limit. 
\end{example}

%
\begin{proof}[Proof of Theorem~\ref{thm:regmain}]
First we show that \eqref{thm:regmain:mod} implies \eqref{thm:regmain:sum}.
Let $M$ be any module of depth $d$ such that $b^R(M)=\lambda \mathbf{v}$.  Since $\chi_{[i,n]}(b^R(M))= \rank \Omega^i(M)$ for $i\geq 0$, the Auslander--Buchsbaum formula
implies that this is strictly positive for $i=1, \dots, n-d$ and $0$ for $i>n-d$.  The proof of Proposition~\ref{prop:moredetailed} then shows
that $b^R(M)$ has the desired form.

Next we show that \eqref{thm:regmain:sum} implies \eqref{thm:regmain:mod}.
If there exists any $M$ such that $b^R(M)=\mathbf{v}$, then the Auslander--Buchsbaum formula implies
that $M$ has depth $d$.  It thus suffices to produce a module $M$ with the desired Betti sequence.  
We may also assume that the coefficient  $a_{-1}$ of $\rhofree$ equals $0$.

 Let $C$ denote the cone spanned by $\rho_0, \dots, \rho_{n-d-1}$, so that 
  $\mathbf{v}$ now belongs to the interior of $C$.
  The proof of Proposition~\ref{prop:moredetailed} illustrates that 
  for each $i=0, \dots, n-d-1$, we can construct $\rho_i$ as the limit
  of Betti sequences of Cohen--Macaulay modules of codimension $n-d$.  
  Since we can construct every extremal ray of $C$ via such a sequence, 
  it follows that every interior point of $C$ can be written as a 
  $\QQ$-convex combination of the Betti sequences of 
  Cohen--Macaulay $R$-modules of codimension $n-d$. In particular,
  $\mathbf{v}$ has this property, and hence $\mathbf{v}\in \BQ(R)$, as
  desired.  This construction also implies the final sentence of the
  theorem, as we have written $\mathbf{v}$ as the sum of Betti
  sequences of Cohen--Macaulay modules of codimension $n-d$. 
\end{proof}

%

\begin{example}[Oscillation of Betti numbers]\label{ex:oscillate}
Let $n=\dim R$ be congruent to $1$ mod $3$. Let $0 < \delta \ll 1$ be a
rational number and set
\[
a_i':=\begin{cases}
0&\text{if } i=-1,\\
1-\frac{\delta}{2} & \text{if } i \ge 0 \text{ and } i \equiv 0 \pmod 3,\\
\frac{\delta}{2} & \text{if } i \ge 0 \text{ and } i \equiv \pm 1 \pmod 3.
\end{cases}
\]
Let $\mathbf{v}':=\sum_{i} a_i'\rho_i$, so that the entries of 
$\mathbf{v}'$ oscillate between $1$ and $\delta$. 
Then there exists a finite length $R$-module
$N$ such that $b^R(N)$ is a scalar multiple of $\mathbf{v}'$. See
Figure~\ref{fig:v}.
\end{example}

\begin{remark}
For a finite length module, the Buchsbaum--Eisenbud--Horrocks Rank Conjecture proposes that 
$b_i(M) \geq \binom{n}{i}$ for $i=0, 1, \dots, n.$ 
It is natural to seek a sharper lower bound $B_i$ that depends on the number 
of generators of $M$ and the dimension of the socle of $M$. 
For $B_1$ we may set
$B_1(b_0,b_n):=b_0-1+n$, and then $b_1\geq B_1(b_0,b_n)$; something similar holds for $B_{n-1}$.  However, Theorem~\ref{thm:regmain}
implies that when $i\ne 1,n-1$ there is no such linear bound. This follows immediately from the fact
that, for any $0< \delta \ll 1$, there is a resolution with shape $(1,1+\frac{\delta}{2},\delta, \dots, \delta,1+\frac{\delta}{2}, 1)$.
\end{remark}

\begin{question}
Are there nonlinear functions $B_i(b_0,b_n)$ such that $b_i(M)\geq B_i(b_0(M),b_n(M))$ for all finite length modules $M$?
\end{question}

\begin{remark}[The graded/local comparison]\label{rmk:comparison}
If 
$S=\Bbbk[x_1, \dots, x_n]$ (with the standard grading) and
$R=\Bbbk[x_1, \dots, x_n]_{(x_1, \dots, x_n)}$,
then there is a map 
$\BQ(S)\to \BQ(R)$ obtained by ``forgetting the grading'' and
localizing.  
Theorem~\ref{thm:regmain} implies that this map is surjective. 
It would be interesting to understand if a similar statement is true 
if we replace $S$ by a more general graded ring.
\end{remark}

\section{Betti sequences over hypersurface rings I:  the cone $\BQtot$}
\label{sec:cone:all:hyp}

We say that $Q$ is a \defi{hypersurface ring} of a regular local ring 
$(R, \fm)$ if $Q = R/\<f\>$ for some nonzerodivisor $f\in R$. 
To avoid trivialities, we assume that $f \in \fm^2$. 
Let $n := \dim R$ and $d := {\rm ord}(f)$, i.e., the unique
integer $d$ such that $f \in \fm^d \minus \fm^{d+1}$. 
The following result is the basis for the ``standard construction.'' 
See~\cite{shamash}, \cite{eisenbud-ci}*{\S 7}, or \cite{avramov-notes} 
for more details.

\begin{theorem}[Eisenbud, Shamash] \label{thm:standardconstruction} 
  Given a $Q$-module $M$, let ${\bf F}_\bullet \to M$ be its minimal
  free resolution over $R$. Then there are maps $s_k \colon {\bf F}_\bullet
  \to {\bf F}_{\bullet + 2k-1}$ for $k \ge 0$ such that
  \begin{enumerate}[\rm (i)]
  \item $s_0$ is the differential of ${\bf F}_\bullet$.
  \item \label{item:firsthomotopy} $s_0s_1 + s_1s_0$ is multiplication
    by $f$.
  \item \label{item:higherhomotopy} $\sum_{i=0}^k s_i
    s_{k-i} = 0$ for all $k > 1$. 
  \end{enumerate}
\end{theorem}

We note that if $R$ and $Q$ are graded local rings, then the maps $s_k$ can
be chosen to be homogeneous. 
Using the $s_k$, we may form a new complex ${\bf F}'_\bullet$
with terms 
\[
{\bf F}'_i = \bigoplus_{j \ge 0} {\bf F}_{i-2j} \otimes_R Q
\]
and with differentials given by taking the sum of the maps 
\begin{align*} {\bf F}_i \otimes_R Q \xrightarrow{(s_0, s_1, s_2,
    \dots)} ({\bf F}_{i-1} \oplus {\bf F}_{i-3} \oplus {\bf F}_{i-5}
  \oplus \cdots) \otimes_R Q.
\end{align*}
Then ${\bf F}'_\bullet \to M$ is a $Q$-free resolution which need not
be minimal.

With $\WW=\prod_{i=0}^\infty \QQ$ and $\epsilon_i\in \WW$ the $i$th
coordinate vector, we define $\Phi \colon \WW \to \WW$ by
\[
\Phi(w_0, w_1, \dots) := 
(w_0, w_1, w_0+w_2, w_1+w_3, w_0+w_2+w_4, \dots).
\]
In other words, the $\ell$th coordinate function of $\Phi$ is given by
\[
\Phi_{\ell}(w_0,w_1,\dots)=\begin{cases}
\sum_{i=0}^{\frac{\ell}{2}} w_{2i} & \text{if $\ell$ is even,}\\
 \vspace{-.3cm}\\
\sum_{i=0}^{\frac{\ell-1}{2}} w_{2i+1} & \text{if $\ell$ is odd.}
\end{cases}
\]
As in Section~\ref{sec:cone:reg}, let $\rhofree:=\epsilon_0$ and 
$\rho_i:=\epsilon_i+\epsilon_{i+1}$ for $i\geq 0$. 

Free resolutions over a hypersurface ring can be infinite in length,
but they are periodic after $n$ steps~\cite{eisenbud-ci}*{Corollary~6.2}, so that $b^Q_i(M)=b^Q_{i+1}(M)$ for all $i\geq
n$~\cite{eisenbud-ci}*{Proposition~5.3}.  Thus, if we seek to describe
the cone of Betti sequences in the hypersurface case, it is necessary
to include some rays with infinite support. We define
\[
\tau^\infty_{i}:=\sum_{j=i}^\infty \epsilon_j\in \WW
\]
and note that $\tau^\infty_i=\Phi(\rho_i)$.  
The rays $\tau^\infty_{n-2}$ and $\tau^\infty_{n-1}$ 
will be especially important for us.

We now give a precise description of the 
total hypersurface cone $\BQtot$ from Definition~\ref{defn:Rtot}.  

\begin{proposition}\label{prop:moreDetailedHyper}
The following three $(n+1)$-dimensional cones in $\WW$ coincide: 
\begin{enumerate}[\rm (i)]\newcounter{savenum}
\item \label{item:full hyp cone} 
	The total hypersurface cone ${\BQtot}$. 
\item \label{item:hyperextremal} 
	The cone spanned by the rays 
	$\QQ_{\geq 0}\langle \rhofree, \rho_0, \dots, \rho_{n-2}, 
	\tau^\infty_{n-2}, \tau^\infty_{n-1}\rangle$.
\item \label{item:hyperfacets} 
	The cone defined by the functionals 
  \[
  \begin{cases}
    \chi_{[i,j]}\geq 0 & \text{ for all } i\leq j \le n \text{ with }
    i-j \text{ even},\\ 
    \chi_{[i,i+1]}=0 & \text{ for all } i\geq n, \text{ and}\\
    \chi_{[n-1,n]}\geq 0.\\
  \end{cases}
  \]
\end{enumerate}
\end{proposition}

\begin{proof}
  It is straightforward to check that the extremal rays satisfy the
  desired facet inequalities, and hence we have
  \eqref{item:hyperextremal}$ \subseteq $\eqref{item:hyperfacets}.
  The reverse inclusion is more difficult than the analogous statement
  in Proposition~\ref{prop:moredetailed} because
  here~\eqref{item:hyperextremal} is not a simplicial cone. We first
  identify the boundary facets, and then show that for each boundary
  facet, one of the listed functionals vanishes on it.

  To do this, we use that these rays satisfy a unique linear
  dependence relation.  When $n$ is even, the relation is given by
\[
\tau^{\infty}_{n-1}+\rho_{n-3}+\dots +\rho_{-1}=\tau^{\infty}_{n-2}+\rho_{n-4}+\dots +\rho_0,
\]
and a similar relation holds when $n$ is odd.
We now consider subsets of these rays 
of size $n$, which we index by the two rays that we omit from the collection. 
These fall into three categories:
\begin{enumerate}[\rm (a)]
\item \label{item:functionalA} $\{ \rho_i,\rho_j\}$ with $i<j$,
\item \label{item:functionalB} $\{\rho_i,\tau^{\infty}_j\}$, and 
\item \label{item:functionalC}
  $\{\tau^{\infty}_{n-2},\tau^{\infty}_{n-1}\}$.
\end{enumerate}
Any such collection is linearly independent, and hence spans a unique
hyperplane of the subspace 
\begin{align}\label{eq:subspace}
\{ w \in \WW \mid w_{n+i} = w_n \text{ for
  all } i\geq0 \}.
\end{align}
As such, there is a unique up to scalar functional vanishing on each
collection; we write $F_{i,j}$ for the corresponding functional in
type \eqref{item:functionalA}, $G_{i,j}$ for type
\eqref{item:functionalB}, and $H$ for type \eqref{item:functionalC}.
In order to show the desired containment, we compute these functionals
and determine which correspond to boundary facets
of~\eqref{item:hyperextremal} by evaluating the functionals on their
corresponding omitted rays.

To begin, note that if $j<n-2$, then $F_{i,j}=\chi_{[i+1,j]}$. 
This evaluates to $1$ on $\rho_i$ for $i\geq 0$, 
$(-1)^{j-(i+1)}$ on $\rho_j$, and $0$ on the remaining rays. 
Hence it determines a boundary facet if and only if $i+1$ and $j$ 
have the same parity.
In addition, for any $i<n-2$, $F_{i,n-2} = \chi_{[n-1,n]}$, 
which is the last functional in~\eqref{item:hyperfacets}. 

Next, observe that $G_{i,n-2}=\chi_{[i+1,n]}$.  If $i<n-2$, this
evaluates to $1$ on $\rho_i$ for $i\geq 0$, $(-1)^{n-(i+1)}$ on
$\tau_{n-2}$, and $0$ on the remaining rays.  Hence in this case, it
yields a boundary facet if and only if $i+1$ and $n$ have the same
parity.  Similarly, $G_{i,n-1}=\chi_{[i+1,n-1]}$ if and only if
$i<n-2$, which is a boundary facet only if $n$ and $i$ have the same
parity.

Finally, we compute that $G_{n-2,n-2}=\chi_{[n-1,n]}$,
$G_{n-2,n-1}=\chi_{[n-1,n]}$, and $H=\chi_{[n,n]}$, which all appear
in~\eqref{item:hyperfacets}.  As the subspace
description~\eqref{eq:subspace} accounts for the remaining
functionals, we have established the equivalence
of~\eqref{item:hyperextremal} and~\eqref{item:hyperfacets}.

We next show 
that~\eqref{item:full hyp cone}$\subseteq$\eqref{item:hyperfacets}.  
For this it suffices to check that the functionals in~\eqref{item:hyperfacets} 
are nonnegative on points in $\BQ(Q)$, 
where $Q=R/\<f\>$ and $f\in \fm$ is arbitrary.  
We thus reduce to the consideration of a point $\mathbf{w}=b^{Q}(M)$, 
where $M$ is a $Q$-module. In this case, 
\[
\chi_{[i,j]}(b^Q(M))=\frac{1}{e(Q)}\left(e(\Omega^i(M))+(-1)^{i-j}e(\Omega^j(M))\right),
\]
which is certainly nonnegative when $i$ and $j$ have the same parity.
It follows from~\cite{eisenbud-ci}*{Proposition~5.3, Corollary~6.2}
that $\chi_{[i,i+1]}(b^Q(M))=0$ for $i\geq n$.  Thus it remains to
check the inequality $\chi_{[n-1,n]}(b^Q(M))\geq 0$.  Using $\mu(N)$
to denote the minimal number of generators of a module $N$, we have
\[
\chi_{[n-1,n]}(b^Q(M))=\mu(\Omega^{n-1}(M))-\mu(\Omega^{n}(M)).
\]
Both of these syzygy modules are maximal Cohen--Macaulay 
$Q$-modules. The key difference is that $\Omega^{n-1}(M)$ 
might have a free summand, whereas $\Omega^n(M)$ does not. 
Since maximal Cohen--Macaulay modules without free summands 
over hypersurface rings have a periodic resolution 
by~\cite{eisenbud-ci}*{Theorem~6.1(ii)}, it follows that 
$\chi_{[n-1,n]}(b^Q(M))$ computes the number of free summands in
$\Omega^{n-1}(M)$, so it is nonnegative.

To complete the proof, we show
that~\eqref{item:hyperextremal}$\subseteq$\eqref{item:full hyp cone}
by showing that each extremal ray lies in $\BQtot$.  We first show
that $\rho_i$ belongs to $\overline{\BQ(R/\<f\>)}$ for any $f$. Choose
a regular local subring $R'\subseteq R/\<f\>$ of dimension $n-1$ and
an $R'$-module $M'$. Then
$b^{R/\<f\>}(M'\otimes_{R'}R/\<f\>)=b^{R'}(M')$ because $R/\<f\>$ is
finite and flat over $R'$.  In particular,
$\overline{\BQ(R')}\subseteq\overline{\BQ(R/\<f\>)}$.  Since $\rho_i
\in \overline{\BQ(R')}$ by Proposition~\ref{prop:moredetailed}, we
have $\rho_i \in \overline{\BQ(R/\<f\>)}$.

Finally, we must show that $\tau^{\infty}_{n-2}$ and
$\tau^{\infty}_{n-1}$ belong to $\BQtot$. 
This is where the advantage of working with $\BQtot$ becomes clear, 
as it enables a second limiting argument that, roughly speaking, 
makes the standard construction exact.  
The key observation is summarized in Lemma~\ref{lem:phiexact} below.

In fact, we now show the more general statement that $\Phi(\rho_i)\in
\BQtot$ for $i=0,\dots, n-1$.  Fix $i$ and let $d^{i,t}$ be the
sequence of degree sequences defined in the proof of
Proposition~\ref{prop:moredetailed}.  
For each $t$, we choose any polynomial 
$f_t\in \mathfrak m^{d^{i,t}_n-d^{i,t}_0+1}$.  
We now apply Lemma~\ref{lem:phiexact}, 
along with the fact that $\Phi$ is continuous, to conclude that
\begin{align*}
  \tau^{\infty}_i&=\Phi(\rho_i)\\
  &= \Phi \left( \lim_{t\to \infty} b^{R}(M(d^{i,t})\otimes_S R) \right)\\
  &= \lim_{t\to \infty} \Phi  \left( b^{R}(M(d^{i,t})\otimes_S R) \right)\\
  &= \lim_{t\to \infty} b^{R/\<f_t\>}(M(d^{i,t})\otimes_S R).
\end{align*}
Since $b^{R/\<f_t\>}(M(d^{i,t})\otimes_S R) \in \BQtot$ for all $t$,
it follows that the final limit lies in $\BQtot$.
\end{proof}

\begin{remark}\label{rmk:equals asymptotic}
  The proof of Proposition~\ref{prop:moreDetailedHyper} goes through
  if we replace $\BQtot$ by the closure of the limit cone $\lim_{t\to
    \infty} \BQ(R/\<f_t\>)$, illustrating that these two cones are
  equal as well.  This justifies equation~\eqref{eqn:equals
    asymptotic}.
\end{remark}

\begin{lemma}\label{lem:phiexact}
  Let $M$ be an $R$-module that is annihilated by
  $\fm^{N_0}$ and let $f\in \fm^N$ with $N\gg N_0$.  Then
  \[
  \Phi(b^R(M))=b^{R/\<f\>}(M).
  \]
  More specifically, let $d=(d_0,\dots,d_n)$ be a degree sequence, 
  $M(d)\otimes_S R$ be defined as in Proposition~\ref{prop:pureoverRLR}, and 
  $f\in \fm^{d_n-d_0+1}$. Then
  \[
  \Phi(b^R(M(d)\otimes_S R))=b^{R/\<f\>}(M(d)\otimes_S R).
  \]
\end{lemma}
\begin{proof}
  Since $R$ is a regular local ring, the minimal $R$-free resolution
  of $M$ has finite length. So there are only finitely many $j$ such
  that the $s_j$ in Theorem~\ref{thm:standardconstruction} are
  nonzero, and there is some positive integer $P$ such that the matrix
  entries in the minimal $R$-free resolution of $M$ belong to
  $\fm^P$. To conclude, we need to know that the entries of each $s_j$
  belong to the maximal ideal $\fm$. From
  Theorem~\ref{thm:standardconstruction}\eqref{item:higherhomotopy},
  this will be true if it holds for $j=1$, and this in turn is true if
  we set $N_0 = P$ and apply
  Theorem~\ref{thm:standardconstruction}\eqref{item:firsthomotopy}.
\end{proof}

\begin{remark} 
Assume that $n\geq 3$. By \cite{triangulations}*{Lemma~2.4.2}, there are 
exactly two triangulations of the cone $\BQtot$, which we now describe. 
First, we project from $\WW$ onto the first $n+1$ coordinates. 
This does not change the combinatorial structure of the cone.
The hyperplane section of the projection given by 
$\epsilon_{0} + \cdots + \epsilon_{n} = 1$ 
is an $n$-dimensional polytope with vertices 
$\rho_{-1}$, $\frac{1}{2}\rho_0$, $\frac{1}{2}\rho_1$, $\dots$, 
$\frac{1}{2}\rho_{n-2}$, $\frac{1}{3}\tau^\infty_{n-2}$, 
$\frac{1}{2}\tau^\infty_{n-1}$. 

To express the triangulations, let $\Delta_r$ denote the polytope 
generated by all vertices other than $r$. 
If $n$ is odd, then the two triangulations are
  \[
  \{ \Delta_{\rho_i} \mid i \text{ odd}, i\ne n-2 \} \cup \{
  \Delta_{\tau^\infty_{n-1}} \}\quad \text{ or } \quad \{ \Delta_{\rho_i} \mid i
  \text{ even} \} \cup \{\Delta_{\tau^\infty_{n-2}}\}.
  \]
   If $n$ is even, then the two triangulations are
  \[
  \{ \Delta_{\rho_i} \mid i \text{ odd} \} \cup \{
  \Delta_{\tau^\infty_{n-2}} \},\quad \text{ or } \quad \{ \Delta_{\rho_i} \mid i
  \text{ even}, i \ne n-2 \} \cup \{\Delta_{\tau^\infty_{n-1}}\}. \qedhere
  \]
\end{remark}

\section{Betti sequences over hypersurface rings II:  A fixed hypersurface}
\label{sec:cone:one:hyp}

For a regular local ring $(R,\fm)$ and $f\in \fm_R$, the cone $\BQtot$ 
is larger than $\BQ(Q)$ for the hypersurface ring $Q = R/\<f\>$. 
In this section, we seek to make this relationship precise.  
Set $Q:=R/\<f\>$ and $d := {\rm ord}(f)$, 
i.e., $f \in \fm^d \setminus \fm^{d-1}$. We note that $e(Q)=d$.
We define the vectors 
\[
\tau^d_{n-2}:= \left( \tfrac{d-1}{d}\epsilon_{n-2}+\sum_{j=n-1}^\infty \epsilon_j \right)
\qquad \text{ and } \qquad 
\tau^d_{n-1}:=\left( \tfrac{1}{d}\epsilon_{n-2}+\sum_{\ell=n-1}^\infty
    \epsilon_\ell \right). 
\]
We also define the functionals
\[
\xi^d_{[i,j]}:=
\begin{cases}
-\epsilon_j^*+d\chi_{[i,j-1]}&\text{ if } i-j \text{ is odd,}\\
(d-1)\epsilon_j^*+d\chi_{[i,j-1]}&\text{ if } i-j \text{ is even}.\\
\end{cases}
\]





The following proposition gives some partial information about
Conjecture~\ref{conj:hyper:equiv}. 

\begin{proposition}\label{prop:one direction}
  The following two $(n+1)$-dimensional cones in $\WW$ coincide:
\begin{enumerate}[\rm (i)]
\item\label{conj:cone:rays} 
	The cone spanned by the rays 
	$\QQ_{\geq 0}\langle \rhofree, \rho_0, \dots, \rho_{n-2}, 
	\tau^d_{n-2}, \tau^d_{n-1}\rangle$.
\item\label{conj:cone:facets}  
	The cone defined by the functionals 
	\[
	\begin{cases}
	\xi^d_{[i,n]}\geq 0 & \text{for all } 0\leq i \leq n,\\
	\chi_{[i,j]}\geq 0 & \text{for all } i\leq j\leq n \text{ and } i-j \text{ even,}\\
	\chi_{[i,i+1]}=0 & \text{for all } i\geq n,  \text{ and}\\
	\chi_{[n-1,n]}\geq 0. &\\
	\end{cases}
	\]
	\end{enumerate}
Furthermore, this cone contains $\overline{\BQ(Q)}$.
\end{proposition}

%
\begin{proof}
One may check that the cones \eqref{conj:cone:rays} and \eqref{conj:cone:facets} coincide by an argument entirely analogous
to that used in the proof of Proposition~\ref{prop:moreDetailedHyper}.
It thus suffices to check that the functionals in~\eqref{conj:cone:facets} 
are satisfied by all points in $\BQ(Q)$. 
By applying Proposition~\ref{prop:moreDetailedHyper}, 
we immediately reduce to the case of showing that $\xi^d_{[i,n]}$ 
is nonnegative on any Betti sequence $b^Q(M)$. 

Fix a finitely generated $Q$-module $M$ and a minimal resolution of $M$:
$0 \leftarrow M \leftarrow Q^{b_0} \leftarrow Q^{b_1} \leftarrow \cdots$. 
To compute $\xi^d_{[i,n]}(b^Q(M))$, we consider the exact sequence
\[
\xymatrix{0&\Omega^i(M)\ar[l]&Q^{b_i}\ar[l]&Q^{b_{i+1}}\ar[l]&\dots \ar[l]&Q^{b_n}\ar[l]&\Omega^{n+1}(M)\ar[l]&0\ar[l]
}.
\]
Assume now that $n-i$ is even and that $i\geq 1$. 
Taking multiplicities, we obtain the equation
\begin{align*}
  e(\Omega^i(M))+e(Q^{b_{i+1}})
  +\dots+ e(Q^{b_{n-1}})+e(\Omega^{n+1}(M)) &=
  e(Q^{b_{i}})+e(Q^{b_{i+2}})+\dots+ e(Q^{b_{n}}),
\end{align*}
which can be rewritten as
\begin{align*}
  e(\Omega^i(M))&= d\chi_{[i,n]} \left( b^Q(M)\right)-e(\Omega^{n+1}(M)).\\
  \intertext{Since $\Omega^{n+1}(M)$ is Cohen--Macaulay, 
  $e\left(\Omega^{n+1}(M)\right)\geq
    \mu\left(\Omega^{n+1}(M)\right)=b_{n+1}^Q(M)=b_{n}(M).$ Hence}
  e(\Omega^i(M))&\leq d\chi_{[i,n]} \left( b^Q(M)\right)-b_{n}(M)
  =\xi^d_{[i,n]}\left( b^Q(M)\right).
\end{align*}
It follows that $\xi^d_{[i,n]}\left( b^Q(M)\right)$ is nonnegative, as desired.

When $n-i$ is odd and $i\geq 1$, essentially the same argument holds,
starting instead from the exact sequence
\[
\xymatrix{0&\Omega^i(M)\ar[l]&Q^{b_i}\ar[l]&Q^{b_{i+1}}\ar[l]&\dots
  \ar[l]&Q^{b_{n-1}}\ar[l]&\Omega^{n}(M)\ar[l]&0\ar[l] }.
\]
The same argument also holds when $i=0$, after one replaces 
$e(\Omega^i(M))$ by the number
\[
e':=\begin{cases}
e(M) & \text{ if } \dim(M)=\dim(Q),\\
0& \text{ otherwise}. 
\end{cases} 
\qedhere
\]
\end{proof}
 
The opposite inclusion also holds when $Q$ has embedding dimension $2$.

\begin{proposition} If $Q$ is a hypersurface ring of embedding
  dimension $2$, then $\BQ(Q)$ satisfies
  Conjecture~\ref{conj:hyper:equiv}.
\end{proposition}
\begin{proof} 
 By Proposition~\ref{prop:one direction}, it suffices to show that the
  desired extremal rays lie in $\overline{\BQ(Q)}$.  We may quickly
  reduce to showing that $\tau_0^d, \tau_1^d\in \overline{\BQ(Q)}$.
 Let $\fm_Q$ denote the maximal ideal of $Q$, $Q':=Q/\fm_Q^{d-1}$, and
$\omega_{Q'}$ be its canonical module. A direct computation confirms that
  $
  d \tau_1^d=b^Q(Q')$ and $
  d\tau_0^d=b^Q\left( \omega_{Q'}\right).
  $
%
\end{proof}

\begin{remark}[Codimension 2 complete intersections]\label{rmk:codim2}
  For arbitrary quotient rings $Q$ of a regular local ring $R$, the
  cone of Betti sequences $\BQ(Q)$ need not be finite dimensional.  For
  instance, consider $Q=\QQ[[x,y]]/\<f_1,f_2\>$ for any regular
  sequence $f_1, f_2$ inside $\<x,y\>^2$.  Let $\mathbf{T}_\bullet$ be the
  Tate resolution of the residue field of $Q$. Since $Q$ is Gorenstein, and
  hence self-injective, we may construct a doubly infinite acyclic
  complex $\mathbf{F}_\bullet$ as below:
\[
\mathbf{F}_\bullet\colon \quad 
\cdots \longleftarrow \mathbf{T}_1^* \longleftarrow \mathbf{T}_0^* 
\longleftarrow \mathbf{T}_0 \longleftarrow 
\mathbf{T}_1 \longleftarrow \mathbf{T}_2 \longleftarrow \cdots.
\smallskip
\]
For all $i\geq 0$, let $M_i$ be the kernel of $\mathbf{T}^*_i\to
\mathbf{T}^*_{i+1}$, and set $\tau_i:=b^Q(M_i)$. The $\tau_i$ are
linearly independent since $\rank \mathbf{T}_i=i+1$ for all $i$
(see~\cite{avramov-buchweitz}*{Example~4.2} for details). So we see
that $\BQ(Q)$ is infinite dimensional.  In particular, $\BQ(Q)$ is
spanned by infinitely many extremal rays.
\end{remark}


\end{document}